\documentclass [12pt] {article}

\usepackage{hyperref}

\usepackage[title]{appendix}
\usepackage{fancyhdr}

\usepackage{amsmath}

\usepackage{amsfonts}
\usepackage{amssymb}

\usepackage{natbib}

\usepackage{enumerate}

\usepackage{amsthm}
\theoremstyle{definition}
\newtheorem {definition} {Definition} [section]
\newtheorem {defn} [definition] {Definition}

\newtheorem {lemma} [definition] {Lemma}
\newtheorem {theorem} [definition] {Theorem}
\newtheorem* {theorem*} {Theorem}
\newtheorem {proposition} [definition] {Proposition}
\newtheorem {corollary} [definition] {Corollary}
\newtheorem {fact} [definition] {Fact}

\theoremstyle{remark} 

\newtheorem*{remark} {Remark}

\DeclareMathOperator {\acl} {acl}
\DeclareMathOperator {\dcl} {dcl}

\newcommand {\etp} {\operatorname {tp}_{\exists}}

\newcommand{\ACF}{\ensuremath{\mathrm{ACF}}}
\DeclareMathOperator {\I} {I}

\DeclareMathOperator {\eacl} {\acl_{\exists}}
\DeclareMathOperator {\edcl} {\dcl_{\exists}}

\DeclareMathOperator {\Diag} {Diag}
\newcommand {\liff} {\leftrightarrow}
\DeclareMathOperator {\etheory} {Th_\exists}
\newcommand {\EA} {\mathrm {EA}}
\newcommand {\sopone} {\mathrm {SOP_1}}
\newcommand {\nsopone} {\mathrm {NSOP_1}}
\newcommand {\tptwo} {\mathrm {TP_2}}
\DeclareMathOperator {\Aut} {Aut}

\newcommand {\G} {\mathbb {G}}

\newcommand {\Q} {\mathbb {Q}}
\newcommand {\N} {\mathbb {N}}

\newcommand{\ga}{\ensuremath{\mathbb G}_{\mathrm{a}}}
\newcommand{\gm}{\ensuremath{\mathbb G}_{\mathrm{m}}}

\makeatletter
\newcommand {\forksym} {\raise0.2ex\hbox{\ooalign{\hidewidth$\vert$\hidewidth\cr\raise-0.9ex\hbox{$\smile$}}}}
\def\@forksim@#1#2#3#4{\mathrel{\mathop{\forksym}\displaylimits_{#2}^{#4}}}
\def\@forksym@#1#2{\@ifnextchar^{\@forksim@#1#2}{\mathrel{\mathop{\forksym}\displaylimits_{#2}}}}
\def\nforks{\@ifnextchar_{\@forksym@}{\forksym}}
\newcommand {\forks} {\not \nforks}
\makeatother

\DeclareMathOperator{\td}{td}

\DeclareMathOperator{\Emb}{Emb}
\newcommand{\ECcat}{\ensuremath{\mathcal{EC}}}

\newcommand{\Lering}{\ensuremath{L_{\mathrm{E\text{-}ring}}}}
\newcommand{\Tefield}{\ensuremath{T_{\mathrm{E\text{-}field}}}}

\newcommand{\tuple}[1]{\ensuremath{\langle #1 \rangle}}
\renewcommand{\le}{\leqslant}

\newcommand{\subs}{\subseteq}

\newcommand{\M}{\mathfrak M}

\newcommand{\entails}{\vdash}
\newcommand{\alg}{\ensuremath{\mathrm{alg}}}

\newcommand{\EAgen}[1]{\ensuremath{\langle #1 \rangle^{\EA}}}
\newcommand{\minus}{\smallsetminus}

\newcommand{\abar}{{\ensuremath{\bar{a}}}}
\newcommand{\bbar}{{\ensuremath{\bar{b}}}}

\newcommand{\dbar}{{\ensuremath{\bar{d}}}}

\newcommand{\xbar}{{\ensuremath{\bar{x}}}}
\newcommand{\ybar}{{\ensuremath{\bar{y}}}}

\newcommand{\JEPrefinement}{JEP-refinement}

\newcommand{\comp}[1]{\widehat {#1}}

\makeatletter
\def\@hangfrom#1{\setbox\@tempboxa\hbox{{#1}}%
      \hangindent 1.5em
      \noindent\box\@tempboxa}
\makeatother

\title{Existentially Closed Exponential Fields}
\author{Levon Haykazyan and Jonathan Kirby}
\begin{document}
\maketitle

\begin{abstract}

 We characterise the existentially closed models of the theory of exponential fields. They do not form an elementary class, but can be studied using positive logic. We find the amalgamation bases and characterise the types over them.
We define a notion of independence and show that independent systems of higher
dimension can also be amalgamated. We extend some notions from classification theory to positive logic and position the category of existentially closed exponential fields in the stability hierarchy as $\nsopone$ but $\tptwo$.
\end{abstract}

\section {Introduction}

An {\em exponential field} is a field $F$ with a homomorphism $E$ from its additive
group $\G_a(F)$ to its multiplicative group $\G_m(F)$.
 Exponential fields are axiomatised by an inductive theory
$\Tefield$.


The classical approach to the model theory of algebraic structures begins by
identifying extensions and finding the model companion. Such an approach for
exponential algebra was initiated by \cite{van-den-dries-erings}. However there
is no model companion which traditionally meant a dead end for model theory. So
the model theory of exponentiation developed in other directions, notably
\cite{wilkie-expansions} proved that the real exponential field is model complete and
\cite{zilber-pseudo-exp} studied exponential fields which are \emph{exponentially closed}, which is similar to existential closedness but only takes account of so-called \emph{strong} extensions.

However, it is possible to study the model theory and stability theory of the
existentially closed models of an inductive theory even when there is no model
companion. \cite{Shelah_lazy} called this setting model theory of \emph{Kind~3},
where Kind~1 is the usual setting of a complete first-order theory, and Kind~2,
also known as \emph{Robinson theories}, is like Kind~3 but where there is
amalgamation over every subset of every model. (There is also Kind~4, which is
now known as homogeneous abstract elementary classes.) More recently there has
been interest in positive model theory, which is only slightly more general
than Kind~3. 

This is the approach taken in this paper, and it means that we look at embeddings and existentially definable sets instead of elementary embeddings and all first-order definable sets. Following \cite{pillay-ec-forking} we call it the \emph{Category of existentially closed models} rather than Kind~3.

\subsection*{Overview of the paper}

In Section~2 we give the necessary background on the category $\ECcat(T)$ of existentially closed models of an inductive theory $T$.

In Section~3 we characterise the existentially closed models of $\Tefield$ in Theorem
\ref{acfe-axioms}. They are the exponential fields $F$ where each additively free
subvariety $V$ of $\G_a(F)^n \times \G_m(F)^n$ has an exponential point, that is,
there is a tuple $\bar a \in F$ such that $(\bar a, E(\bar a)) \in V$. Additive
freeness of a variety, however, is not a first-order property and we show that
the class of existentially closed exponential fields is not axiomatisable by a first-order theory.

The theory $\Tefield$ of exponential fields has neither the joint embedding
property nor the amalgamation property. In Section~4 we show that the amalgamation holds if and only if
the base is algebraically closed (Theorem \ref{2-amalgamation}). The absence of
the joint embedding property means that the category of exponential fields and
embeddings splits into a disjoint union of classes where any two structures in
the same class can be jointly embedded into a third one. We show that there are
$2^{\aleph_0}$ classes and characterise the theories of each class (Corollary
\ref{completion-corollary}).

In Section~5 we define a notion of independence: if $F$ is an algebraically closed
exponential field, $A, B, C \subseteq F$, then we say that $A$ is independent of
$B$ over $C$ and write $A \nforks_C B$ if $\EAgen{AC}$ and $\EAgen{BC}$ are
field-theoretically independent over $\EAgen{C}$. Here $\EAgen{X}$ is the
smallest algebraically closed exponential subfield containing $X$. We then show that
independent systems of higher dimension can be amalgamated (Theorem
\ref{n-amalgamation}). In particular, we have 3-amalgamation for independent systems, which is sometimes known as \emph{the independence theorem}.

In Section~6 we study two tree properties of formulas in $\Tefield$: the
tree property of the second kind, $\tptwo$, and the $1$-strong order property,
$\sopone$.  Both properties are well-known and extensively studied for complete first-order theories.
Our definitions are variants that specialise to the
appropriate definitions for model-complete theories. We show that the category $\ECcat(\Tefield)$ has
$\tptwo$ (Proposition \ref{have-tptwo}) by exhibiting a $\tptwo$-formula. We then show that the independence notion we defined is well-enough behaved to apply the analogue of a theorem of Chernikov and Ramsey, and deduce that $\ECcat(\Tefield)$ is $\nsopone$ (Theorem
\ref{no-sopone}). 

In an appendix, we give the necessary technical details of this generalised stability theory for the category of existentially closed models of an inductive theory. In particular, we  adapt the work of \cite{chernikov-ramsey-tree-properties} on $\nsopone$-theories.

\subsection* {Acknowledgement}

The first named author would like to thank R\'emi Jaoui and Rahim Moosa for many
discussions around the topics of this paper.


\section{Model-theoretic background}

We give some background on the model theory of the category of existentially closed models of a (usually incomplete) inductive first-order theory. 
More details can be found in \cite{hodges-big}, \cite{pillay-ec-forking} and
\cite{ben-yaacov-poizat-positive}. The only novelty in this section is the
notion of a \JEPrefinement\  of an
inductive theory $T$ (Definition \ref{jep-refinement-defn}), which is a useful syntactic counterpart to the choice of a
monster model.

\begin{remark}
The readers familiar with positive logic will notice that the axioms for
exponential fields are not only inductive, but {\em h-inductive} in the sense of
\cite{ben-yaacov-poizat-positive}. So we can treat exponential fields in
positive logic, which amounts to working with homomorphisms rather than
embeddings. In our case this makes no difference, since homomorphisms of fields
are embeddings. This is the case whenever all atomic formulas have positive
negations modulo the theory. So the setting here is equivalent to positive model
theory, with the extra assumption that all atomic formulas have positive
negations modulo the theory. This extra assumption, however, is not essential,
and everything in this section applies to a general h-inductive theory (although
the treatment of algebraic and definable closures would require more care).
\end{remark}

\subsection{Existentially closed models of an inductive theory}

Recall that a first-order theory $T$ is \emph{inductive} if the union of a chain
of models is a model. Equivalently, $T$ is axiomatised by
$\forall\exists$-sentences. We work in the category $\Emb(T)$ whose objects are
models of $T$ and whose arrows are embeddings. 

\begin{defn}
A model $M \models T$ is \emph{existentially closed} if for all quantifier-free formulas $\phi(\xbar,\ybar)$ and all $\abar$ in $M$, if there is an extension $M \subs B$ such that $B \models T$ and $B \models \exists\xbar \phi(\xbar,\abar)$ then $M \models \exists\xbar \phi(\xbar,\abar)$.
\end{defn}

If $T$ is inductive then for any $A \models T$ it is straightforward to build an extension $A \subseteq M$ by a transfinite induction process such that $M$ is an existentially closed model of $T$.

Embeddings between structures preserve $\exists$-formulas. That is if $f : A \to
B$ is an embedding, then for any $\exists$-formula $\phi(\bar x)$ and $\bar a
\in A$ we have 
\[A \models \phi(\bar a) \implies B \models \phi(f(\bar a)).\]
If the converse of the above implication holds too, then the embedding is called
an {\em immersion}.

Given $A \models T$ and $\abar$ a tuple from $A$, the existential type $\etp^A(\abar)$ is the set of existential formulas $\phi(\xbar)$ such that $A \models \phi(\abar)$. If $A \subs B$ then $\etp^A(\abar) \subs \etp^B(\abar)$, and if the inclusion is an immersion we have equality.

The following well-known equivalent characterisations of existentially closed models are useful.
\begin{fact}\label{ec models fact}
Let $T$ be an inductive theory and let $M \models T$. Then the following conditions are equivalent.
\begin{enumerate}[(i)]
\item $M$ is an existentially closed model of $T$.
\item For any model $B \models T$ any embedding $f : M \to B$ is an immersion.
\item For every $\bar a \in M$, the $\exists$-type $\etp^M(\bar a)$ is maximal, that is, if $B \models T$ and $\bbar \in B$ are such that $\etp^M(\abar) \subs \etp^B(\bbar)$ then $\etp^B(\bbar) = \etp^M(\abar)$.
\item For every $\bar a \in M$ and $\exists$-formula $\phi(\bar x)$ such that $M
\models \lnot \phi(\bar a)$, there is an $\exists$-formula $\psi(\bar x)$ such
that $M \models \psi(\bar a)$ and $T \models \lnot \exists \bar x (\phi(\bar x)
\land \psi(\bar x))$.
\end{enumerate}
\end{fact}
\begin{proof}
The equivalence of (i) and (ii) is \cite[Theorem~8.5.6]{hodges-big}, and the equivalence of (i) and (iv) is \cite[Corollary~3.2.4]{hodges-bmg}. The equivalence of (ii) and (iii) is immediate.
\end{proof}

We write $\ECcat(T)$ for the full subcategory of $\Emb(T)$ consisting of the existentially closed models of $T$ and all embeddings between them (which by the above fact are immersions).

The model theory and stability theory of existentially closed models of an inductive $T$ is developed analogously to that of a complete first-order theory $T'$. Existentially-closed models of $T$ correspond to models of $T'$; immersions correspond to elementary embeddings, and models of $T$ correspond to substructures of models of $T'$.

Note however that the category $\ECcat(T)$ does not determine the theory $T$
completely. We recall the following well-known equivalence.

\begin{fact}
\label{companion-fact}
For two inductive theories $T_1$ and $T_2$ the following are equivalent.
\begin{enumerate}[(i)]
\item The theories $T_1$ and $T_2$ have the same universal consequences.
\item Every model of $T_1$ embeds in a model of $T_2$ and vice-versa.
\item The existentially closed models of $T_1$ and $T_2$ are the same.
\end{enumerate}
\end{fact}
\begin{proof}
The equivalence of (i) and (ii) is well known. For the equivalence of (i) and
(iii) see \cite[Theorem 3.2.3]{hodges-bmg}.
\end{proof}

Theories satisfying these equivalent conditions are called {\em companions}.
Thus $\ECcat(T)$ (as a subcategory of all $L$-structures and embeddings)
determines $T$ only modulo companions.

\subsection{Amalgamation bases}

\begin{defn}
An \emph{amalgamation base} for $\Emb(T)$ is a model $A \models T$ such that given any two models $B_1, B_2 \models T$ and embeddings
$f_1 : A \to B_1$ and $f_2 : A \to B_2$ there is a model $C \models T$ and
embeddings $g_1 : B_1 \to C$ and $g_2 : B_2 \to C$ such that $g_1f_1 = g_2f_2$.

$A$ is a \emph{disjoint amalgamation base} if furthermore we can pick the embeddings $g_1$ and $g_2$ in such a way that  $g_1(B_1) \cap g_2(B_2) = g_1f_1(A)$. (Some authors including Hodges call this a \emph{strong amalgamation base}.)
\end{defn}
\begin{fact}[{\cite[Corollary~8.6.2]{hodges-big}}]
Every existentially closed model of an inductive theory $T$ is a disjoint amalgamation base for $\Emb(T)$.
\end{fact}
However there can be amalgamation bases which are not existentially closed models. 

There is also a connection with algebraically closed sets.
\begin{defn}
Let $M$ be an existentially closed model of an inductive theory $T$ and let $A
\subseteq M$. Then $\eacl(A)$ is the union of all finite $\exists$-definable
subsets of $M$ using parameters from $A$. 
If $A = \eacl(A)$ we say that $A$ is \emph{$\exists$-algebraically closed} in $M$.
\end{defn}
 It is not immediately obvious that
$\eacl(\eacl(A)) = \eacl(A)$ and that $\eacl(A)$ is the same if calculated in an
existentially closed extension of $M$. This follows from the fact that if an
existential formula $\phi(x, \bar a)$ defines a finite set of size $n$ in $M$,
then the formula
$$\exists x_1, \dots , x_{n+1} \left[\bigwedge_{i = 1}^{n+1} \phi(x_i, \bar a)
\land \bigwedge_{i \neq j} x_i \neq x_j\right]$$
is false in $M$. It follows by Fact~\ref{ec models fact}(iv) that there must be an
existential formula $\chi(\bar y)$ that implies this and holds of $\bar a$.

\begin{fact}[{\cite[Corollary~8.6.8]{hodges-big}}]\label{dab = eacl}
Let $T$ be an inductive theory and let $A$ be an amalgamation base for $\Emb(T)$. Then $A$ is a disjoint amalgamation base if and only if $A$ is $\exists$-algebraically closed (in an existentially closed extension $M$ of $A$).
\end{fact}

Amalgamation bases are also very useful for understanding the existential types which are realised in existentially closed models of $T$. Indeed, finding the amalgamation bases plays much the same role as proving a quantifier-elimination theorem for a complete first-order theory.
\begin{proposition}
Let  $M$, $N$ be existentially closed models of an inductive theory $T$, and suppose $\abar \in M$ and $\bbar \in N$ are tuples. 
Suppose also that there are amalgamation bases $A \subs M$ and $B \subs N$ with $\abar \in A$ and $\bbar \in B$, and an isomorphism $\theta: A \to B$ such that $\theta(\abar) = \bbar$. Then $\etp^M(\abar) = \etp^N(\bbar)$.
\end{proposition}
\begin{proof}
We have $\etp^A(\abar) = \etp^B(\bbar)$. Then by \cite[Theorem~8.6.6]{hodges-big}, since $A$ and $B$ are amalgamation bases, there is a unique way to extend this type to a maximal $\exists$-type in $T$, which by Fact~\ref{ec models fact}(iii) must be $\etp^M(\abar)$ and $\etp^N(\bbar)$.
\end{proof}
So the existential type of a tuple $\abar$ in an existentially closed model is determined by how $\abar$ embeds into an amalgamation base.

\subsection{JEP}

\begin{defn}
The category $\Emb(T)$ has the {\em joint embedding property (JEP)} if for any
two models of $T$ there is a third in which they can both be embedded. We also
say that $T$ has the JEP in that case. 
\end{defn}

We have the following characterisation of JEP.

\begin{lemma}
\label{jep-characterisation-lemma}
The category $\Emb(T)$ has the JEP if and only if for every pair of universal
sentences $\phi$ and $\psi$, if $T \entails \phi \lor \psi$ then $T \entails
\phi$ or $T \entails \psi$.
\end{lemma}
\begin{proof}
An easy argument using the method of diagrams and compactness. See
\cite[Exercise 3.2.8]{hodges-bmg}.
\end{proof}

When looking at existentially closed models of $T$, extending $T$ to an inductive theory with JEP plays the role of choosing a completion of $T$. However, not any extension will do. For example, if $T$ is the theory of fields, the existentially closed models of $T$ are the algebraically closed fields. The completions are then given by fixing the characteristic. However we could also extend $T$ to the inductive theory $T'$ of orderable fields (fields in which $-1$ is not a sum of squares). Then the existentially closed models of $T'$ are real-closed fields.

We now make a new definition.
\begin{defn}
\label{jep-refinement-defn}
An inductive extension $T'$ of an inductive theory $T$ is called a
\emph{\JEPrefinement} of $T$ if $T'$ has the joint embedding property and every
existentially closed model of $T'$ is an existentially closed model of $T$.
\end{defn}

We can use amalgamation bases to find these \JEPrefinement s.
\begin{lemma}\label{amalg base JEP}
If $A$ is an amalgamation base for $\Emb(T)$ then $T \cup \etheory(A)$ is a
\JEPrefinement\ of $T$.
\end{lemma}
\begin{proof}
It is clear that $T \cup \Diag(A)$ has the JEP, where $\Diag(A)$ is the
quantifier-free diagram of $A$. However $T \cup \Diag(A)$ and $T \cup
\etheory(A)$ have the same consequences in the language of $T$. It follows from
Lemma \ref{jep-characterisation-lemma} that $T \cup \etheory(A)$ has the JEP
too.

Then since $T \cup \etheory(A)$ is an extension of $T$ by existential sentences
only, it is easy to see that every existentially closed model of $T \cup
\etheory(A)$ is an existentially closed model of $T$.
\end{proof}

Note that each existentially closed model $M$ of $T$ is a model of a unique
\JEPrefinement\ of $T$ modulo companions. Indeed $M \models T \cup \etheory(M)$
which is a \JEPrefinement\ of $T$. Conversely if $T'$ is a \JEPrefinement\ of
$T$ and $M \models T'$, then for every universal sentence $\phi$ we have $T'
\entails \phi$ if and only if $M \models \phi$. Indeed if $M \models \phi$, then
there is an existential sentence $\psi$ such that $T \entails \lnot \psi \lor
\phi$ and $M \models \psi$ (by Fact \ref{ec models fact}). Then $T' \not
\entails \lnot \psi$ and therefore, by JEP, $T' \entails \phi$. This shows that
the universal consequences of $T'$ are completely determined by $M$.

\subsection{Monster models}

As for complete first-order theories, it is notationally convenient (though not essential) to
work inside a \emph{monster model} $\mathfrak M$ of an inductive theory $T$ with
the JEP, that is, a model of some large cardinality $\kappa$ which is both
$\kappa^+$-universal and strongly $\kappa$-homogeneous.
By \emph{$\kappa^+$-universal} we mean that  every $A \models T$ with $|A| < \kappa^+$ embeds in $\M$. By \emph{strongly $\kappa$-homogeneous} we means that if $A$ is an amalgamation base for $\Emb(T)$,
and $f_1$, $f_2$ are embeddings of $A$ into $\M$, then there is an automorphism
$\theta$ of $\M$ such that $\theta \circ f_1 = f_2$.
 Then all models
considered are submodels of $\mathfrak M$, and maximal (existential) types are
the same as orbits of $\Aut(\mathfrak M)$. In this setting, monster models are
often called universal domains. \cite{pillay-ec-forking} calls them
\emph{e-universal domains}. As both universality and homogeneity are important,
we prefer the terminology \emph{monster model}. 

As for complete first-order theories, the universality and homogeneity properties together are equivalent to a saturation property:
 any existential type $p(\bar x)$ using parameters from a
 set $A$ of cardinality less than $\kappa$ that has a realisation in an
extension of $\mathfrak M$ already has a realisation in $\mathfrak M$. It follows that monster models of $T$ are existentially closed.


As usual, the existence of monster models depends on stability-theoretic conditions on $T$ or set-theoretic conditions on $\kappa$, for example, that $\kappa$ is strongly inaccessible. However since everything could be done in the category $\Emb(T)$ or $\ECcat(T)$, albeit at some cost in notation, we will not worry about moving outside ZFC like this. The only place we really use monster models in this paper is where we introduce independence relations in the usual way, as certain relations on subsets of the monster model. The alternative of treating them as relations on commuting squares in $\Emb(T)$ seems not to be widely known.

If $T$ does not have the JEP, then by choosing a monster model $\mathfrak M
\models T$ we are in effect choosing a \JEPrefinement\ of $T$ (modulo
companions).

\section {Existentially closed exponential fields}

In this section we characterise the existentially closed exponential fields. For the basics on exponential rings and fields the reader can consult
\cite{van-den-dries-erings}, \cite{macintyre-ealgebra} or \cite{kirby-fpef}.

\begin{definition}
An {\em exponential field} (or {\em E-field} for short) is a field $F$ of
characteristic zero, together with a homomorphism $E$ from the additive group
$\G_a(F)$ to the multiplicative group $\G_m(F)$.

If the field is algebraically closed we call it an \emph{EA-field}.
\end{definition}

The reason for excluding characteristic $p > 0$ is that in that case for any
element $x$ we have
$$(E(x) - 1)^p = E(x)^p - 1 = E(px) - 1 = E(0) - 1 = 0,$$ 
and there are no non-trivial nilpotents, so $E(x) = 1$.

We work in the category of exponential fields and their embeddings. Model
theoretically this means that we use the language $\Lering = \tuple{+,-,\cdot,
0, 1, E}$ of $E$-rings, where $E$ is a unary function symbol, and we look at the
class of models of the theory $\Tefield$ axiomatised by
\begin{enumerate}[(i)]
\item the axioms of fields of characteristic $0$;
\item $\forall x, y [E(x + y) = E(x) \cdot E(y)]$;
\item $E(0) = 1$.
\end{enumerate}

Atomic $\Lering$-formulas are exponential polynomial equations, so an exponential field $F$ is {\em existentially closed} if every finite system of exponential polynomial equations and inequations (with coefficients from $F$) that has a solution in an extension of $F$, already has a solution in $F$.

We will repeatedly use the following well known result.

\begin{fact} [see e.g. \cite{fuchs-abelian}]\label{divisibility fact}
Divisible Abelian groups are injective in the category of Abelian groups. That is,
if $A, B, Q$ are Abelian groups, $A \le B$ and $Q$ is divisible, then any
homomorphism $f : A \to Q$ extends to $B$.
\end{fact}

From this we immediately derive the following result.

\begin{proposition}
Every exponential field extends to an algebraically closed exponential
field.
\end{proposition}

\begin{proof}
Let $F$ be an exponential field and let $F'$ be the algebraic closure of $F$.
Then $\G_a(F)$ is a subgroup of $\G_a(F')$ and $\G_m(F)$ is a subgroup of
$\G_m(F')$. So we can view $E$ as a homomorphism from $\G_a(F)$ to $\G_m(F')$.
But since $F'$ is algebraically closed, its group of units is divisible. Hence,
by Fact~\ref{divisibility fact}, $E$ extends to a homomorphism from $\G_a(F')$
to $\G_m(F')$.
\end{proof}

It follows that existentially closed exponential fields are algebraically closed
fields. The converse of course is not true as being algebraically closed says
nothing about the solubility of exponential equations. We will characterise
existentially closed exponential rings geometrically, via \emph{Pillay-Pierce}-style axioms. 

\begin{defn}
Let $F$ be an exponential field. Let $V \subseteq \G_a(F)^n \times
\G_m(F)^n$ be a subvariety which is irreducible over $F$, and let $(\bar x, \bar y)$ be a point of $V$ in a field extension of $F$, which is generic in $V$ over $F$. We say that $V$ is {\em additively free} if $\bar x$ satisfies no
equation of the form $\sum_{i = 1}^n m_i x_i = a$ where $a \in F$ and $m_i \in
\mathbb Z$, not all zero.
\end{defn}

\begin{theorem}
\label{acfe-axioms}
Let $F$ be an exponential field. Then $F$ is existentially
closed if and only if for every additively free subvariety $V \subseteq
\G_a(F)^n \times \G_m(F)^n$ there is a point $\bar a \in F$ such that $(\bar a,
E(\bar a)) \in V$.
\end{theorem}

\begin{proof}
We first show the left to right implication. Suppose $V \subseteq \G_a(F)^n
\times \G_m(F)^n$ is additively free. Let $F'$ be an algebraically closed field
extension of $F$ containing a point $(\abar, \bbar) \in V$, generic over $F$,
where $\abar = (a_1,\dots,a_n)$ and $\bbar = (b_1,\dots,b_n)$. We extend the
exponential homomorphism $E : \G_a(F) \to \G_m(F)$ to $\G_a(F')$ as follows.
Define $E'(a_i) = b_i$ for $i=1,\dots,n$.  Since $V$ is additively free and
$(\abar,\bbar)$ is generic in $V$ over $F$, the $a_i$ are $\Q$-linearly
independent over $F$. So this assigment extends uniquely to a homomorphism $E'$
extending $E$, defined on the subgroup generated by $F$ and $\abar$. Now since
$\G_m(F')$ is divisible we can extend $E'$ to $\G_a(F')$. Thus we have an
exponential field extending $F$ where $V$ has an exponential point. By
existential closedness of $F$, there is an exponential point already in $F$. 

Conversely, assume that every additively free subvariety has an exponential
point in $F$. By taking $V \subs \G_a \times \G_m$ given by $p(y) = 0$, a single polynomial in the $\G_m$-coordinate, we see that $F$ is algebraically closed.

To show existential closedness we need to show that every system of
exponential equations and inequations that has a solution in an extension of $F$
already has a solution in $F$. Two easy observations help to simplify this.
Firstly, inequations can be reduced to equations as 
$$\Tefield \entails x \neq 0 \liff \exists y (xy = 1).$$ 
Secondly, iterated exponentials can be substituted by additional variables as, for example,
$$\Tefield \entails f(\bar x, E(\bar x), E^2(\bar x)) = 0 \liff \exists \bar y [f(\bar x, \bar y,
E(\bar y)) = 0 \land \bar y = E(\bar x)].$$
Thus to show that $F$ is existentially closed it is enough to show that every
finite system of polynomial equations $P(\bar x, \bar y) = 0$ that has an exponential solution
$P(\bar a, E(\bar a)) = 0$ in an extension $F' \supseteq F$,
already has an exponential solution in $F$. So let $P$, $F'$ and $\bar a =
(a_1, \dots , a_n)$ be as above. Let $V$ be the locus of $(\abar, E(\abar))$ over $F$, that is, the smallest subvariety of $\ga^n \times \gm^n$ containing $(\abar,E(\abar))$ which is defined over $F$.
If $V$ is additively free, we will get an exponential point in $F$ by our assumption.
However this may not be the case. Without loss of generality assume that $a_1,
\dots , a_k$ are $\Q$-linearly independent over $F$ and $a_{k+1}, ..,
a_n$ are in the $\Q$-linear span $\langle Fa_1, \dots , a_k
\rangle_{\mathbb Q}$.  Thus there is a $n$-tuple $\bar b \in F$ and an $n \times
k$ matrix $A$ of rational numbers such that
$$
\begin{bmatrix}
a_1 \\
\vdots \\
a_n
\end{bmatrix}
= A
\begin{bmatrix}
a_1 \\
\vdots \\
a_k
\end{bmatrix}
+ \bar b.
$$
Let $N$ be the least common multiple of the denominators of entries of $A$. Now
let $V'$ be the locus of $(\frac {a_1} N, \dots , \frac {a_k} N, E(\frac {a_1} N),
\dots , E(\frac {a_k} N))$ over $F$. 
Then $V'$ is additively free.
By the assumption, there is a point $\bar c \in F$ such that $(\bar c, E(\bar c))
\in V'$. Let 
$$
\bar d = A
\begin{bmatrix}
N c_1 \\
\vdots \\
N c_k \\
\end{bmatrix}
+ \bar b.
$$
Then $(\dbar,E(\dbar)) \in V(F)$, and in particular $P(\dbar,E(\dbar)) = 0$.
So $F$ is existentially closed.
\end{proof}

We can derive a number of consequences from this characterisation.

\begin{corollary}
If $F$ is an existentially closed exponential field, then the homomorphism $E :
\G_a(F) \to \G_m(F)$ is surjective.
\end{corollary}
\begin{proof}
For any $a \in F^\times$ the variety $W(x,y)$ given by $y = a$ is additively free and
therefore has an exponential point.
\end{proof}

Next we show that $\mathbb Z$ is universally definable in any existentially
closed exponential field as the multiplicative stabiliser of the kernel.

\begin{corollary}
If $F$ is an existentially closed exponential field, then for every $a \in F$
$$a \in \mathbb Z \text { iff } F \models \forall x [E(x) = 1 \to E(ax) = 1].$$
\end{corollary}

\begin{proof}
The left to right implication is clear. For the converse assume that $c \in F
\minus \mathbb Z$.

If $c \in \mathbb Q$, then write $c = \frac n m$ where $m > 1$ and $n,m$ are
coprime. By the previous corollary there is $b \in F$ such that $E(b)$ is a
primitive $m$-th root of unity. Then $E(mb) = 1$ but $E(cmb) = E(nb) = E(b)^n
\neq 1$.

If $c \not \in \mathbb Q$, then pick $d \in F$ that is distinct from $1$ and
consider the variety $W(x_1, x_2, y_1, y_2)$ defined by the equations
$$cx_1 = x_2, y_1 = 1, y_2 = d.$$
Since $c \not \in \mathbb Q$, the variety $W$ is additively free and therefore
has an exponential point $(a, b, E(a), E(b)) \in W$. But then $E(a) = 1$ and
$E(ca) = E(b) = d \neq 1$.
\end{proof}

Since a countable set cannot be definable in a saturated model of a first-order
theory we also have the following consequence.

\begin{corollary}
The class of existentially closed exponential fields is not first-order
axiomatisable. Equivalently, the theory $\Tefield$ does not have a
model-companion. \qed
\end{corollary}

\begin{remark}
Note that the class of existentially closed exponential fields is axiomatisable
by an $L_{\omega_1,\omega}$-sentence.

Several papers of Boris Zilber and of the second author are devoted to the study of \emph{exponentially closed fields}. These are not existentially closed in the sense of this paper, but they are existentially closed within the category of exponential fields and so-called \emph{strong} extensions. To get the JEP we also insist that the exponential fields have the Schanuel property, and usually that they have standard kernel. Then the strong extensions are roughly those which preserve the Schanuel property and the kernel. 

The axiom giving the exponential closedness is then quite similar to our axiom
giving existential closedness, specifying that certain varieties have
exponential points on them. In that case, the varieties are not just additively
free but also \emph{multiplicatively free} (to avoid extending the kernel) and
\emph{rotund} (to preserve the Schanuel property). These axioms are also
$L_{\omega_1,\omega}$-expressible.
\end{remark}

\section {Amalgamation bases}

In this section we characterise the amalgamation bases for exponential fields.
Here and later we will use the independence notion for algebraically closed fields.
\begin{definition}\label{ACF indep defn}
Let $K$ be an algebraically closed field, and $A,B,C$ subfields of $K$ with $C \subseteq A$ and $C \subseteq B$. We say that $A$ is independent from $B$ over $C$ (with respect to the theory $\ACF$ and as subfields of $K$) and write $A \nforks_C^{\ACF} B$ if for every finite tuple $\bar a$ from $A$, the transcendence degree satisfies $\td(\bar a/B) = \td(\bar a/C)$.
\end{definition}

The following is a very well-known special case of the uniqueness of non-forking extensions in stable theories.
\begin{fact}\label{ACF amalg}
Suppose $C$ is an algebraically closed field and $A$, $B$ are two field extensions of $C$. Then there is, up to isomorphism, a unique way to embed $A$ and $B$ into an extension field $K$ such that  $A \nforks_C^{\ACF} B$, and $K$ is generated as a field by (the image of) $A \cup B$.
\end{fact}

\begin{theorem}
\label {2-amalgamation}\label{amalg base is EA-field}
The amalgamation bases for $\Emb(\Tefield)$ are precisely the algebraically closed exponential fields (EA-fields). Furthermore they are disjoint amalgamation bases.
\end{theorem}

\begin{proof}
Let $F$ be an EA-field, and let $f_1 : F \to F_1$ and $f_2 : F \to F_2$ be two
$E$-field extensions of $F$. Using Fact~\ref{ACF amalg} there is an
algebraically closed field $K$ and field embeddings $g_1 : F_1 \to K$ band $g_2
: F_2 \to K$ such that $g_1f_1 = g_2f_2$ and $F_1 \nforks_F^{\ACF} F_2$. By
identifying $F, F_1, F_2$ with their images in $K$ we can assume that all
embeddings $f_i, g_i$ are inclusions.

We would like to extend $E_1 \cup E_2$ to a homomorphism from $\G_a(K)$ to
$\G_m(K)$. Note that $E_1 \cup E_2$ extends to a homomorphism from the group
$F_1 + F_2$ generated by $F_1$ and $F_2$ to $\G_m(K)$.  Indeed, it is enough to
see that it is well defined. Assume that $a_i, a_i' \in F_i$ and $a_1 + a_2 =
a_1' + a_2'$. But then $a_1 - a_1' = a_2' - a_2 \in F_1 \cap F_2$. By
$\ACF$-independence, $F_1 \cap F_2 = F$. It follows that $E_1(a_1 - a_1') =
E_2(a_2' - a_2)$ and therefore 
$$\frac {E_1(a_1)} {E_1(a_1')} = \frac {E_2(a_2')} {E_2(a_2)}.$$
But this implies that $E_1(a_1)E_2(a_2) = E_1(a_1')E_2(a_2')$.

The final step is to extend this homomorphism from $F_1 + F_2$ to $\G_m(K)$ by
the divisibility of the latter, using Fact~\ref{divisibility fact}. So EA-fields
are disjoint amalgamation bases.

For the converse, suppose that $F$ is an E-field which is not an EA-field. Let
$F_1 = F^\alg$. We can extend the exponential map on $F$ to some exponential map
$E_1$ on $F_1$, and in particular for every $a$ which is algebraic over $F$ we
have that $E_1(a)$ is algebraic over $F$. Now let $a \in F^\alg \minus F$ and
let $F_2 = F(a,t)$, with $t$ transcendental over $F$. Then there is an
exponential map $E_2$ on $F_2$ extending that on $F$ such that $E_2(a) = t$.
Then $F_1$ and $F_2$ cannot be amalgamated over $F$.
\end{proof}

The fact that algebraically closed exponential fields are disjoint amalgamation bases allows us
to characterise $\eacl$ in existentially closed models.

\begin{defn}
Suppose that $F$ is an EA-field, and $A \subs F$. Write $\EAgen{A}_F$  (or just $\EAgen{A}$ if $F$ is clear) for the smallest EA-subfield of $F$ containing $A$, that is, the intersection of all the EA-subfields of $F$ containing $A$.
\end{defn}

\begin{corollary}\label{eacl characterization}
Let $F$ be an existentially closed exponential field and $A \subseteq F$. Then
$\eacl(A) = \EAgen{A}_F$.
\end{corollary}
\begin{proof}
It is clear that $\eacl$-closed sets are algebraically closed $E$-fields. The
converse follows immediately from Theorem \ref{amalg base is EA-field} and
Fact \ref{dab = eacl}.
\end{proof}





This allows us to characterise the \JEPrefinement s of $\Tefield$.

\begin{corollary}
\label{completion-corollary}
\begin{enumerate}[(i)]

\item If $F$ is an EA-field, then $\Tefield \cup \etheory(F)$ is a
\JEPrefinement\ of $\Tefield$. 

\item There are $2^{\aleph_0}$ many \JEPrefinement s of $\Tefield$ (modulo
companions), corresponding to the minimal EA-fields (those EA-fields $F$ such
that $F = \EAgen{0}_F$).
\end{enumerate}
\end{corollary}

\begin{proof}
Part (i) follows from Lemma~\ref{amalg base JEP}.

For part (ii), every existentially closed exponential field $K$ is a model of
$\Tefield \cup \etheory(F)$ where $F = \EAgen{0}_K = \EAgen{0}_F$. On the other
hand if $F_1$ and $F_2$ are non-isomorphic minimal EA-fields, then they cannot
be jointly embedded in an E-field. It follows that $\Tefield \cup \etheory(F_1)$
and $\Tefield \cup \etheory(F_2)$ are not companions.

Any minimal EA-field must be countable, so there are at most $2^{\aleph_0}$ of
them. To construct $2^{\aleph_0}$ non-isomorphic ones, note that for each
infinite sequence $(q_n)_{n \in \N, n \ge 2}$ of nonzero rational numbers, there
is an E-field where $E(1)$ is transcendental and $E(E(1)^n) = q_n$ for $n \ge
2$.
\end{proof}

\begin{remark}
There is a similarity with the completions of the theory ACFA of existentially closed fields with an automorphism. There the completions are given by specifying the action of the automorphism on the algebraic closure of the empty set. The main difference here is that the underlying field of a minimal EA-field can have any countable transcendence degree, not just 0 as in ACFA.
\end{remark}

Finally in this section we can use amalgamation bases to give an easy proof of instability. Of course we prove stronger results later.
\begin{corollary}
Let $T$ be any \JEPrefinement\ of \Tefield. Let $A$ be any subset of any model
of $T$. Then there are $2^{\aleph_0 + |A|}$ maximal existential types over $A$
which are realised in $\ECcat(T)$. In particular $\ECcat(T)$ is not stable (in
the sense of type-counting).
\end{corollary}
\begin{proof}
Given $A$ and a variable $x$ we can specify that $x$ is transcendental over $A$
and then we can specify the values of $E(x^n)$ in $(A \cup \Q) \minus \{0\}$
independently for each $n \in \N^+$. Each one of these extends to a distinct
maximal existential type. This gives the maximum number of types.
\end{proof}

\section{Independence and higher amalgamation}

We now define a notion of independence and proceed to establish higher
amalgamation of independent systems. 
\begin{definition}
\label{independence-definition}
Let $F$ be an EA-field. Let $A, B, C \subseteq F$
be subsets.
Recall that $\langle A \rangle^{\EA}$ is the smallest EA-subfield of $F$ containing $A$, and $\EAgen{AC}$ means $\EAgen{A \cup C}$.

We say that $A$ and $B$ are {\em independent} over $C$ in $F$ and write $A
\nforks_C B$ if $\EAgen{AC}$ and $\EAgen{BC}$ are \ACF-independent over $\EAgen{C}$.

\end{definition}

\begin{remark}
This notion of independence is quite weak. For example, it does not look at the behaviour of \emph{logarithms} of elements of $A$, $B$, or $C$, that is, points whose exponentials lie in those subsets. Nor does it look at how $A$ and $B$ might otherwise be related in $\EAgen{ABC}$. However it is appropriate for this setting. Stronger independence notions, suitable for exponentially closed fields, were explored in the thesis of \cite{HendersonThesis}.
\end{remark}

We now introduce the relevant definitions for higher amalgamation. More details can
be found for example in \cite{gkk-amalgamation-functors} or \citet[Chapter
XII]{shelah}.

Let $n \ge 3$ be an integer, which we view as a set, that is, $n = \{0, \dots ,
n-1\}$. We view $\mathcal P(n)$ and $\mathcal P^-(n) = \mathcal P(n)
\minus n$ as categories where there is a unique morphism from $a$ to $b$ if $a
\subseteq b$. We consider functors from these categories to $\Emb(\Tefield)$.
Such functors are called $\mathcal P(n)$ and $\mathcal P^-(n)$-systems
respectively. Given a system $F$ and $a \in \mathcal P(n)$, we denote by $F_a$
the exponential field associated to $a$. For $a \subseteq b$ there is an
embedding $F_a \to F_b$. We normally think of this embedding as being an
inclusion, that is, that $F_a \subseteq F_b$. We say that $F$ is {\em
independent} if for every $a \subseteq b$ we have $$F_a \nforks_{\bigcup_{c
\subsetneq a} F_c} \bigcup_{a \not \subseteq d \subseteq b} F_d$$ as subsets of
the exponential field $F_b$. As an example, in an independent $\mathcal
P(3)$-system we have in particular $F_{\{0,1\}} \nforks_{F_{\{0\}}F_{\{1\}}}
F_{\{0, 2\}} F_{\{1, 2\}}$ as subsets of $F_{\{0, 1, 2\}}$.

If we forget the exponential structure of the given independent system, we end
up with an independent system of algebraically closed fields.  Such systems --
more generally independent systems in stable and simple theories -- have been
extensively studied. We mention two facts concerning such systems. Both are
special cases for \ACF\ of theorems about independent amalgamations in stable
theories.

\begin{fact} [Folklore, see e.g. \cite{dpkm-group-configuration}]
\label{acf-n-amalgamation}
Any independent $\mathcal P^-(n)$-system of algebraically closed fields can be
completed to an independent $\mathcal P(n)$-system.
\end{fact}

\begin{fact} [Fact XII.2.5 of \citet{shelah}]
\label{generalised-heir}
Let $F$ be an independent $\mathcal P(n)$-system of algebraically closed fields
and $t \subseteq n$.  For $i = 1, \dots , m$ let $s(i) \in \mathcal P(n)$ and
let $\bar a_i \in F_{s(i)}$. Assume that for some formula $\phi$ we have $F_n
\models \phi(\bar a_1, \dots , \bar a_m)$. Then there are $\bar a_i' \in F_{s(i)
\cap t}$ such that $F_n \models \phi(\bar a_1', \dots , \bar a_m')$ and if $s(i)
\subseteq t$, then $\bar a_i' = \bar a_i$.
\end{fact}

We can now extend amalgamation to independent $\mathcal P^-(n)$-systems of
EA-fields.

\begin{theorem}[$n$-amalgamation]
\label{n-amalgamation}
Any independent $\mathcal P^-(n)$-system of EA-fields can be completed to an independent $\mathcal P(n)$-system.
\end{theorem}
\begin{proof}
Let $F$ be an independent $\mathcal P^-(n)$-system of EA-fields. For $a \subsetneq n$, let
$E_a$ denote the exponentiation on $F_a$. In this proof it is convenient to use
a notation for complements. So we denote $\comp {i} = n \minus \{i\}$ and
$\comp {i,j} = n \minus \{i, j\}$.

By Fact \ref{acf-n-amalgamation} there is an algebraically closed field $F_n$
and embeddings (which we think of as inclusions) completing $F$ to an
independent $\mathcal P(n)$-system of algebraically closed fields. We need to
extend $E_{\comp 0} \cup \dots \cup E_{\comp {n-1}}$ to a homomorphism
from $\G_a(F_n)$ to $\G_m(F_n)$. There is a unique way to extend it to the group
$F_{\comp 0} + \dots + F_{\comp {n-1}}$ generated by their domains. We
show that this is a well-defined homomorphism. Then by Fact~\ref{divisibility
fact} it can be extended to $\G_a(F_n)$.  To show that this map is well-defined
reduces to showing that if $a_{\comp 0} \in F_{\comp 0}, \dots, a_{\comp {n-1}}
\in F_{\comp {n-1}}$ and $a_{\comp 0} + \dots + a_{\comp {n-1}} = 0$, then it is
the case that $E_{\comp 0}(a_{\comp 0}) \cdots E_{\comp {n-1}}(a_{\comp {n-1}})
= 1$.

We prove by induction on $k$ that if $a_{\comp 0} \in F_{\comp 0}, \dots, a_{\comp
k} \in F_{\comp k}$ and $a_{\comp 0} + \dots + a_{\comp k} = 0$, then $E_{\comp
0}(a_{\comp 0}) \cdots E_{\comp k}(a_{\comp k}) = 1$. For $k = 0$ this is clear.
So assume it holds for $k$. 
Let $a_{\comp 0} \in F_{\comp 0}, \dots, a_{\comp k}
\in F_{\comp k}, a_{\comp {k+1}} \in F_{\comp {k+1}}$ be such that 
$$a_{\comp 0} + \dots + a_{\comp k} + a_{\comp {k+1}} = 0.$$
By Fact \ref{generalised-heir} there are $b_{\comp{0,k+1}} \in F_{\comp
{0,k+1}}, \dots, b_{\comp{k,k+1}} \in F_{\comp{k,k+1}}$ such that
$$b_{\comp{0,k+1}} + \dots + b_{\comp{k,k+1}} + a_{\comp{k+1}} = 0.$$
But since $b_{\comp{0,k+1}}, \dots, b_{\comp{k,k+1}} \in F_{\comp{k+1}}$ we have
\begin{eqnarray*}
\prod_{j=0}^k E_{\comp{j,k+1}}(b_{\comp{j,k+1}}) \cdot E_{\comp{k+1}}(a_{\comp{k+1}}) 
&=& \prod_{j=0}^k E_{\comp{k+1}}(b_{\comp{j,k+1}}) \cdot E_{\comp{k+1}}(a_{\comp{k+1}}) \\
&=&  E_{\comp{k+1}}\left(\sum_{j=0}^k b_{\comp{j,k+1}} + a_{\comp{k+1}}\right) \\ 
&=& E_{\comp{k+1}}(0) = 1.
\end{eqnarray*}

But also
$$a_{\comp 0} - b_{\comp{0,k+1}} + \dots + a_{\comp k} - b_{\comp{k,k+1}} = a_{\comp{k+1}} - a_{\comp{k+1}} = 0.$$
Therefore by the induction hypothesis we conclude that
$$E_{\comp 0}(a_{\comp 0}) \cdots E_{\comp k}(a_{\comp k}) =
E_{\comp{0,k+1}}(b_{\comp{0,k+1}}) \cdots E_{\comp{k,k+1}}(b_{\comp{k,k+1}}).$$
Putting these two together we conclude that
$$E_{\comp 0}(a_{\comp 0}) \cdots E_{\comp k}(a_{\comp k})
E_{\comp{k+1}}(a_{\comp{k+1}}) = 1.$$
This completes the induction and taking $k = n-1$ gives us precisely the
statement that we need.
\end{proof}

\begin{remark}
The conclusion of this theorem is sometimes called \emph{$n$-existence}. Note that the exponential map on $F_n$ is far from uniquely determined, and so the property called \emph{$n$-uniqueness} does not hold. 
This is in contrast to the situation for the category of exponentially closed fields (with the Schanuel property, standard kernel, and the countable closure property) as studied in the papers by \cite{zilber-pseudo-exp} and \cite{BaysKirby18}. In that case both $n$-existence and $n$-uniqueness hold for all $n \in \N$, and indeed this is the core of the proof of the existence and uniqueness of models of each uncountable cardinality.
\end{remark}

\section {Model theoretic tree properties}
\label{tree-properties-section}

In this section we introduce and study two properties of formulas: $\tptwo$ and
$\sopone$. Both properties have been extensively studied in the literature for
complete first-order theories. Our setting of existentially closed models of inductive theories is somewhat more general,
however the results generally transfer, with small changes to the definitions and proofs. 

\subsection{TP${}_2$}

\begin{definition}
Let $T$ be an inductive theory with the JEP.
An existential formula $\phi(\bar x, \bar y)$ has the {\em tree property of the
second kind} ($\tptwo$ for short) with respect to $\ECcat(T)$  if there is an amalgamation base $A \models T$, an existential
formula $\psi(\bar y_1, \bar y_2)$ and parameters $(\bar a_{i,j})_{i,j <
\omega}$ from $A$ such that the following hold:
\begin{enumerate}[(i)]
\item for all $\sigma \in \omega^\omega$ the set $\{\phi(\bar x, \bar
a_{i,\sigma(i)}) : i < \omega\}$ is consistent, that is, it is realised in some $B \models T$ such that $A \subs B$ (or equivalently it is realised in a monster model).

\item $\psi(\bar y_1, \bar y_2)$ implies that $\phi(\bar x, \bar y_1) \land
\phi(\bar x, \bar y_2)$ is inconsistent, that is,
$$T \entails \lnot \exists \xbar \ybar_1 \ybar_2 [\psi(\ybar_1,\ybar_2) \land
\phi(\xbar, \ybar_1) \land \phi(\xbar, \ybar_2)].$$

\item for every $i, j, k < \omega$, if $j \neq k$ then $A \models \psi(\bar
a_{i,j}, \bar a_{i,k})$.
\end{enumerate}
\end{definition}

In the first-order setting the property $\tptwo$ was introduced by
\cite{shelah-simple} and extensively studied by \cite{chernikov-kaplan-ntp2},
\cite{ben-yaacov-chernikov-ntp2} and \cite{chernikov-ntp2}.

\begin{remark}
In the setting of a complete first-order theory, we may always take $\psi(\bar
y_1, \bar y_2)$ to be the formula $\lnot \exists x [\phi(\bar x, \bar y_1) \land
\phi(\bar x, \bar y_2)]$. So the formula $\psi$ is not mentioned in the
definition, which instead simply insists that $\phi(\bar x, \bar a_{i,j}) \land
\phi(\bar x, \bar a_{i,k})$ is inconsistent.

In our setting $\phi(\bar x, \bar a_{i,j}) \land \phi(\bar x, \bar a_{i,k})$
being inconsistent means that there is an
existential formula $\psi(\bar y_1, \bar y_2)$ satisfied by $\bar a_{i,j}, \bar
a_{i,k}$ that implies that $\phi(\bar x, \bar y_1) \land \phi(\bar x, \bar y_2)$
is inconsistent. However this formula $\psi$ may be different for different
triples $(i,j,k)$. So the point of asking for $\psi$ explicitly is to have a
single one that works for all $i,j,k$. An alternative approach would be just to
stipulate that $\phi(\bar x, \bar a_{i,j}) \land \phi(\bar x, \bar a_{i,k})$ is
inconsistent, but to require the existence of parameters $(\bar a_{i,j})_{i,j < \kappa}$ for sufficiently large
$\kappa$.  Then we could use the Erd\"os-Rado theorem to find a sub-tree of the
parameters for which a single formula $\psi$ suffices.
\end{remark}

\begin{proposition}
\label{have-tptwo}
Let $T$ be a \JEPrefinement\ of $\Tefield$. Then the formula $\phi(x, yz) := E(y \cdot x) = z$ has $\tptwo$ with respect to $\ECcat(T)$.
\end{proposition}

\begin{proof}
Take the formula $\psi(y_1z_1, y_2z_2)$ to be $y_1 = y_2 \land z_1 \neq z_2$. 
It is convenient to choose parameters $(\abar_{i,j})_{1 \le i,j < \omega}$, with indices starting at 1 rather than 0.

Let $F \models T$ be an amalgamation base. So it is algebraically closed, and in particular the $\Q$-linear dimension of $F$ is infinite.
Choose $b_i \in F$ for $1 \le i < \omega$ such that $1, b_1, b_2, \dots$ are $\Q$-linearly independent. For $1\le j < \omega$ choose $c_j \in F$ to be distinct and nonzero. We
let $\bar a_{i,j} = b_ic_j$. It is clear that if $j \neq k$, then $F \models
\psi(b_ic_j, b_ic_k)$. So (iii) holds. Also  (ii) holds because the exponential map is a function.

Let $\sigma \in (\omega\minus\{0\})^{\omega\minus\{0\}}$. By compactness, it remains to show that for any $n$ the
formula $\bigwedge_{i=1}^n E(b_i \cdot x) = c_{\sigma(i)}$ is consistent. But this
follows from Theorem \ref{acfe-axioms}. Indeed, consider the variety $W(x_0, \dots,
x_n, y_0, \dots, y_n)$ defined by the equations
\begin{align*}
x_1 &= b_1 \cdot x_0, \\
    & \vdots \\
x_n & = b_n \cdot x_0, \\
y_1 & = c_{\sigma(1)}, \\
    & \vdots \\
y_n & = c_{\sigma(n)}.
\end{align*}
We claim that $W$ is additively free. Indeed, assume that for some $m_i \in
\mathbb Z$ and $d \in F$ we have $\sum_{i=0}^n m_i x_i - d \in \I(W)$. Then
since $(0, \dots, 0, c_{\sigma(1)},\dots, c_{\sigma(n)}) \in W$ we conclude that $d = 0$. Also
$(1, b_1, \dots, b_n, 1, c_{\sigma(1)},\dots, c_{\sigma(n)}) \in W$ and hence we conclude that
$m_0 + \sum_{i=1}^n m_ib_i = 0$. By the choice of the $b_i$ we must have $m_0 = \dots
= m_n = 0$. This shows that $W$ is additively free. Thus it must have an exponential point in some E-field extension of $F$. The first coordinate of this point realises
$\bigwedge_{i = 1}^n E(b_i \cdot x) = c_{\sigma(i)}$, as required.
\end{proof}

As for a complete first-order theory, the property $\tptwo$ for $\Emb(T)$
implies that dividing (appropriately defined) does not have local character.
This means that all \JEPrefinement s of $\Tefield$ are not simple either in the
sense of \cite{pillay-ec-forking} or in the weaker sense of
\cite{ben-yaacov-simplicity}. We give the details in the appendix.

\subsection{NSOP${}_1$}

\begin{definition}
Let $T$ be an inductive theory with the JEP. 
An existential formula $\phi(\bar x, \bar y)$ has the {\em $1$-strong order
property} ($\sopone$ for short) with respect to $\ECcat(T)$ if there is an amalgamation base $A \models T$, an existential formula $\psi(\bar y_1, \bar y_2)$, and a binary tree of parameters $(\bar a_\eta : \eta
\in 2^{<\omega})$ from $A$ such that the following hold:
\begin{enumerate}[(i)]
\item For every branch $\sigma \in 2^\omega$, the set $\{\phi(\bar x, \bar
a_{\sigma|_n} ): n < \omega\}$ is consistent, that is, realised in some extension $A \subs B$ such that $B \models T$.

\item $\psi(\bar y_1, \bar y_2)$ implies that $\phi(\bar x, \bar y_1) \land
\phi(\bar x, \bar y_2)$ is inconsistent, that is,
$$T \entails \lnot \exists \xbar \ybar_1 \ybar_2 [\psi(\ybar_1,\ybar_2) \land
\phi(\xbar, \ybar_1) \land \phi(\xbar, \ybar_2)].$$

\item For every $\eta, \nu \in 2^{<\omega}$, if $\eta^\frown0\preceq \nu$, then
$A \models \psi(\bar a_{\eta^\frown1}, \bar a_\nu)$.
\end{enumerate}

If no existential formula has $\sopone$, we say that $\ECcat(T)$ is $\nsopone$.
\end{definition}

Here the notation $\eta^\frown i$ denotes the sequence $\eta$ with extra element
$i$ at the end and $\eta \preceq \nu$ means that $\eta$ is an initial segment of
$\nu$. As with $\tptwo$, there is no need for $\psi$ in the full first-order
setting where it can always be taken to be $\lnot \exists \bar x [\phi(\bar x,
\bar y_1) \land \phi(\bar x, \bar y_2)]$.

The property $\sopone$ was introduced by \cite{dzamonja-shelah-maximality}
in the full first-order setting. Its systematic study began in
\cite{chernikov-ramsey-tree-properties} and is presently a very active area. We
will use the following version of a theorem proved in \cite{chernikov-ramsey-tree-properties}
for complete first-order thories. 

\begin{theorem}
\label{chernikov-ramsey-theorem}
Let $T$ be an inductive theory with the JEP and let
$\mathfrak M$ be a monster model for $T$. Assume that there is an $\Aut(\mathfrak
M)$ invariant independence relation $\nforks$ on small subsets of $\mathfrak M$
which satisfies the following properties for any small existentially closed model $M$ and any small
tuples $\abar$, $\bbar$ from $\mathfrak M$:
\begin{enumerate}[(i)]
\item {\em Strong finite character:} if $\bar a \forks_M \bar b$, then there is
an existential formula $\phi(\bar x, \bar b, \bar m) \in \etp(\bar a/\bar bM)$
such that for any $\bar a'$ realising $\phi$, the relation $\bar a' \forks_M
\bar b$ holds;
\item {\em Existence over models:} $\bar a \nforks_M M$ for any tuple $\bar a
\in \mathfrak M$;
\item {\em Monotonicity:} $\bar a \bar a' \nforks_M \bar b \bar b'$ implies
$\bar a \nforks_M \bar b$;
\item {\em Symmetry:} $\bar a \nforks_M \bar b$ implies $\bar b \nforks_M \bar
a$;
\item {\em Independent 3-amalgamation:} 
If $\bar c_1 \nforks_M \bar c_2$, $\bar b_1
\nforks_M \bar c_1$, $\bar b_2 \nforks_M \bar c_2$ and $\bar b_1 \equiv_M \bar
b_2$ then there exists $\bar b$ with $\bar b \equiv_{\bar c_1 M} \bar
b_1$ and $\bar b \equiv_{\bar c_2 M} \bar b_2$.
\end{enumerate}
Then $\ECcat(T)$ is $\nsopone$.
\end{theorem}

\begin{proof}
The proof is identical to the proof of
\cite[Proposition~5.3]{chernikov-ramsey-tree-properties}, except the use of
Proposition 5.2 in that paper is replaced by Proposition
\ref{chernikov-ramsey-5.2} from the Appendix.
\end{proof}

\begin{remark}
The reader familiar with independence relations for simple theories may observe that the properties given here do not imply either base monotonicity or local character.
\end{remark}

\begin{theorem}
\label{no-sopone}
The independence $\nforks$ for exponential fields defined in Definition \ref{independence-definition} satisfies the above five conditions, even over any $\EA$-field, not just over existentially closed E-fields. Consequently, if $T$ is any \JEPrefinement\ of \Tefield\ then $\ECcat(T)$ is $\nsopone$.
\end{theorem}

\begin{proof}
Let $T$ be a \JEPrefinement\ of $\Tefield$, and let $\mathfrak F$ be a monster model for $T$.

Existence over models, monotonicity, and symmetry follow from the same properties for $\ACF$-independence. 
Independent 3-amalgamation follows from Theorem~\ref{n-amalgamation}. Although the two statements of independent 3-amalgamation look different, it is well-known and straightforward to prove that they are equivalent.

We prove strong finite character.

Let $F \subseteq \mathfrak F$ be a small EA-subfield. 
Assume that $\bar a \forks_F \bar b$. Then $\EAgen{F\bar a}$ and $\EAgen{F\bar
b}$ are not \ACF-independent over $F$. So there is a finite tuple
 $\bar \beta \in \EAgen{F\bar b}$ which is algebraically independent over $F$, but not algebraically independent over $\EAgen{F\bar a}$. Let $q(\ybar)$ be a polynomial witnessing the algebraic dependence. By dividing through by some coefficient, we may assume that the coefficient of some non-constant term is 1.
 
 Now write $q(\ybar)$ as $p(\bar \alpha,\ybar)$ where $\bar \alpha$ is a finite tuple from $\EAgen{F\bar a}$ and $p(\bar x, \bar y) \in F[\xbar,\ybar]$.

Let $\psi(\bar x, \bar a) \in \etp^{\mathfrak F}(\bar \alpha/F\bar a)$ be an $\exists$-formula which exhibits the witnesses which show that $\bar\alpha$ is in the EA-closure of $F\abar$. We can assume that for any $\abar'$ and $\bar\alpha'$, if $\mathfrak F \models \psi(\bar\alpha',\abar')$ then $\bar\alpha' \in \EAgen{F\abar'}$.

By Corollary~\ref{eacl characterization}, $\bar \beta$ is in the model-theoretic
algebraic closure of $F \bar b$.  Take  $\chi(\bar y, \bar b) \in
\etp^{\mathfrak F}(\bar \beta/F\bar b)$ to be the $\exists$-formula defining the
smallest finite set containing $\bar\beta$ and defined over $F\bar b$. Then
since $\bar \beta$ is (field-theoretically) algebraically independent over $F$, if $\mathfrak F
\models \chi(\bar\beta',\bar b)$ then $\bar \beta'$ is also algebraically
independent over $F$.


We claim that if $\bar a'$ realises the $\exists$-formula $\phi(\bar z)$ given by 
\[\exists \bar x, \bar y \left[\psi(\bar x, \bar z) \land \chi(\bar y, \bar b) \land
p(\bar x, \bar y) = 0 \right],\]
then $\bar a' \forks_F \bar b$. 
Assume that $\phi(\bar a')$ holds and let $\bar \alpha', \bar \beta' \in \mathfrak F$ be witnesses for $\xbar$ and $\ybar$ respectively.

Then $\bar \beta' \in \EAgen{F\bar b}$ and, by the assumption on $\chi$,
the tuple $\bar \beta'$ is algebraically independent over $F$.
By the assumption on $\psi$ we have $\bar \alpha' \in \EAgen{F\bar a'}$.
So the condition that $p(\bar \alpha',\bar \beta')=0$ implies that
$\td(\bar \beta'/F\alpha') < \td(\bar \beta'/F)$. So $\EAgen{F\bar a'}$ and
$\EAgen{F\bar b}$ are not \ACF-independent over $F$, so $\bar a' \forks_F \bar
b$. So $\nforks$ has strong finite character, and hence satisfies all the conditions of Theorem~\ref{chernikov-ramsey-theorem}.
\end{proof}

\begin{appendices}
\section {Generalised stability for the category of existentially closed models}

In this appendix we give the technicalities of generalised stability theory for the category 
$\ECcat(T)$ of an inductive theory $T$. Everything here applies to positive
model theory as well (where atomic formulas need not have negations). These
results are well known for complete theories, but need some modifications to
work more generally. The main modifications are in the definitions, as we have
already done for $\tptwo$ and $\sopone$.

We fix an inductive theory $T$ with the JEP and let $\mathfrak M$ denote its
monster model. All subsets and tuples are assumed to come from $\mathfrak M$. To
simplify the notation we make no distinction between singletons and tuples. For
simplicity, the tuples are assumed to be finite, but this is not necessary.
The notation $a \equiv_A b$ means that $\etp(a/A) = \etp(b/A)$, or equivalently that
there is an automorphism of $\mathfrak M$ fixing $A$ pointwise and taking $a$ to
$b$. If $I$ is a linear order, then a sequence $(a_i)_{i \in I}$ is called {\em
indiscernible} over $A$ if for every $i_1 < \dots < i_n$ and $j_1 < \dots < j_n$
we have $a_{i_1}\dots a_{i_n} \equiv_A a_{j_1}\dots a_{j_n}$. The Ramsey method of
constructing indiscernibles fails in our setting, but the Erd\H{o}s-Rado method
works to give the following fact:
\begin{fact} [Lemma 3.1 of \cite{pillay-ec-forking}]
\label{erdos-rado-indiscernible}
Fix a set $A$. If $\kappa$ is sufficiently large and $(a_i)_{i < \kappa}$ is any
sequence, then there is a sequence $(b_i)_{i < \omega}$ indiscernible over $A$
such that for every $n < \omega$ there are $i_1 < \dots < i_n < \kappa$ with
$$b_1\dots b_n \equiv_A a_{i_1}\dots a_{i_n}.$$
\end{fact}

\begin{definition}
A partial existential type $\Sigma(x, b)$ {\em divides} over $A$ if there is a
sequence $(b_i)_{i < \omega}$ indiscernible over $A$ in $\etp(b/A)$ such
that $\bigcup_{i < \omega} \Sigma(x, b_i)$ is inconsistent.
\end{definition}

Note that if $\Sigma(x, b)$ divides over $A$, then by compactness there is an
existential formula $\phi(x, b) \in \Sigma(x, b)$ that divides. Dividing of
formulas can be explicitly characterised as follows.

\begin{lemma}
A formula $\phi(x, b)$ divides over $A$ if and only if there is $k < \omega$, an
existential formula $\psi(y_1, ..., y_k)$ and a sequence $(b_i)_{i < \omega}$
satisfying the following conditions:
\begin{enumerate}[(i)]
\item for each $i < \omega$, we have $b_i \equiv_A b$;
\item $\psi(y_1, \dots , y_k)$ implies that $\phi(x, y_1) \land \dots  \land
\phi(x, y_k)$ is inconsistent, i.e.
$$T \entails \lnot \exists x, y_1, \dots , y_k [\psi(y_1, \dots , y_k) \land
\phi(x, y_1) \land \dots  \land \phi(x, y_k)];$$
\item for each $i_1 < \dots  < i_k < \omega$, we have $\M \models \psi(b_{i_1}, \dots , b_{i_k})$.
\end{enumerate}
\end{lemma}

\begin{proof}
If $(b_i)_{i < \omega}$ is the indiscernible sequence witnessing the dividing,
then for some $k < \omega$ and $i_1 < \dots < i_k < \omega$ the formula $\phi(x,
b_{i_1}) \land \dots  \land \phi(x, b_{i_k})$ is inconsistent. Therefore there
must be an existential formula $\psi(y_1, \dots , y_k)$, satisfied by $b_{i_1},
\dots , b_{i_k}$, that implies this. By indiscernibility of $(b_i)_{i <
\omega}$, the formula $\psi$ is satisfied for every $i_1 < \dots < i_k <
\omega$.

Conversely assume that the conditions are satisfied. Then by compactness we can
have an arbitrarily long sequence $(b_i)_{i < \kappa}$ satisfying the same
property. Hence by Fact \ref{erdos-rado-indiscernible} we can extract an
indiscernible sequence witnessing the dividing. 
\end{proof}

\begin{definition}
Let $T$ be an inductive theory with JEP. Then $\ECcat(T)$ is called {\em simple} if dividing has local
character. That is, for every $A$ and $b$ there is a subset $A_0 \subseteq A$ of
cardinality at most $|T|$ such that $\etp(b/A)$ does not divide over $A_0$.
\end{definition}

\begin{remark}
This definition of simplicity is due to \cite{ben-yaacov-simplicity}, which
develops an independence relation based on non-dividing for simple theories. The
approach of \cite{pillay-ec-forking} is slightly different. The notion of
dividing there is extended to forking and simplicity is defined as the local
character of forking. This condition, however, implies that forking is
equivalent to dividing and therefore also implies (but is stronger than)
simplicity in the above sense.
\end{remark}

\begin{proposition}\label{simple implies TP2}
If $\ECcat(T)$ is simple then no existential formula has $\tptwo$.
\end{proposition}
\begin{proof}
Assume that the existential formula $\phi(x, y)$, has $\tptwo$ and let $\psi(y_1, y_2)$ be
the formula witnessing that $\phi(x, y_1) \land \phi(x, y_2)$ is inconsistent.
Fix a regular cardinal $\lambda$. We will construct a type over a set of
cardinality $\lambda$ that divides over every subset of smaller cardinality.

By compactness, for arbitrarily large $\kappa$ we can find parameters
$(a_{i,j})_{i,j < \kappa}$ such that 
\begin{enumerate}[(i)]
\item for every $\sigma \in \kappa^\kappa$ the set $\{\phi(x, a_{i,\sigma(i)}) :
i < \kappa\}$ is consistent, and
\item for every $i, j , k < \kappa$, if $j \neq k$, then $\M \models \psi(a_{i,j},
a_{i,k})$.
\end{enumerate}

Construct a function $\sigma \in \kappa^\lambda$ by induction on $i$ and take
$b_i = a_{i, \sigma(i)}$. Assume that $\sigma|_i$ has been constructed. Consider
the sequence 
$$(\etp(a_{i,j}/\{b_k : k < i\}))_{j < \kappa}.$$
If $\kappa$ is large enough, one of these types has to repeat infinitely often.
Pick $\sigma(i)$ such that $\etp(a_{i,\sigma(i)}/\{b_k : k < i\})$ appears
infinitely often and set $b_i = a_{i,\sigma(i)}$. Then $\phi(x, b_i)$ divides
over $\{b_k : k < i\}$. Now let $b$ realise $\{\phi(x, b_i) : i < \lambda\}$.
Then $\etp(b/\{b_i : i < \lambda\})$ divides over every subset of smaller
cardinality.
\end{proof}

Next we turn our attention to $\nsopone$. We prove Proposition
\ref{chernikov-ramsey-5.2}, which
is the generalisation of Proposition 5.2 of
\cite{chernikov-ramsey-tree-properties} to our setting, and which is used in
Theorem \ref{chernikov-ramsey-theorem}. The proof is essentially the same. The
only differences that are not cosmetic are sidestepping the use of Ramsey's
Theorem via Fact \ref{erdos-rado-indiscernible}, and the use of Skolem
functions.

If $A$ is a subset of an existentially closed model $M$, then by $\edcl(A)$ we
denote the set of all the elements of $M$ which are pointwise existentially
definable over $A$. As with $\eacl$, the operator $\edcl$ is a closure operator.
We first show how to add Skolem functions in our setting. 

\begin{lemma}
\label{skolem-lemma}
Let $T$ be an inductive theory in the language $L$. Then there is an expansion
$L'$ of $L$ and an inductive theory $T'$ extending $T$ such that for every
existentially closed model $M' \models T'$ the following hold
\begin{enumerate}[(i)]
\item $M'|_L$ is an existentially closed model of $T$;
\item if $A \subseteq M'$, then $\edcl(A)$ is an existentially closed model of
$T'$.
\end{enumerate}
\end{lemma}

\begin{proof}
If we add function symbols witnessing existential quantifiers and the
corresponding Skolem axioms in the usual way, then the resulting theory $T'$
will be model complete. So every model of $T'$ will be existentially closed,
and the first condition may fail.

We instead add {\em partial} Skolem functions through relation symbols. More
specifically, for each existential formula $\phi(\bar x, y)$ and each partition
of its free variables into a tuple $\bar x$ and a single variable $y$, we add a
new relation symbol $R_\phi(\bar x, y)$ together with the following axioms:
\begin{enumerate}[(i)]
\item $\forall \bar x, y [R_\phi(\bar x, y) \to \phi(\bar x, y)]$;
\item $\forall \bar x, y_1, y_2 [R_\phi(\bar x, y_1) \land R_\phi(\bar x, y_2)
\to y_1 = y_2]$;
\item $\forall \bar x [\exists y \phi(\bar x, y) \liff \exists y R_\phi(\bar x,
y)]$.
\end{enumerate}
We iterate this process the usual way and take the limit. Let $L'$ denote this
language and $T'$ denote the resulting theory. 

In any model $M'$ of $T'$, the symbol $R_\phi$ is interpreted by a partial
function witnessing $\phi(\bar x, y)$. So if $N$ is an extension of $M'|_L$,
then this partial function can be extended to witness $\phi(\bar x, y)$ in $N$.
Therefore we can interpret the symbols of $L' \minus L$ in such a way that the
resulting structure $N'$ is a model of $T'$ and an extension of $M'$. (This may
not be possible if we add function symbols instead.) Thus, if $M'$ is
existentially closed, any existential formula $\psi(\bar x)$ over $M'$ witnessed
by elements of $N'$ is already witnessed by elements of $M'$. In particular,
taking $\psi$ to be an $L$-formula, we see that $M'|_L$ must be existentially
closed too.

For the second condition, let $A \subseteq M'$ and $B = \edcl(A)$. Let
$\phi(\bar b, y)$ be an existential formula over $B$ and $M' \models \exists y
\phi(\bar b, y)$.  We need to find an element witnessing $y$ in $B$. Then $B$
will be a substructure of $M'$ and so a model of $T'_\forall$ and by Facts
\ref{ec models fact} and \ref {companion-fact}, also an existentially closed
model of $T'$. Since $B = \edcl(A)$, there is an existential formula $\psi(\bar
a, y)$ over $A$ such that $M' \models \forall y (\phi(\bar b, y) \liff \psi(\bar
a, y))$. Then there is $b'$ such that $M' \models R_\psi(\bar a, b')$. But this
$b'$ is $\exists$-definable over $A$ and so is in $B$.
\end{proof}

We write $a \nforks^u_C b$ to mean that $\etp(a/Cb)$ is finitely satisfiable in
$C$. Again, we work in a monster model $\M$ of $T$.

\begin{proposition}
\label{chernikov-ramsey-5.2}
If $\phi(x, y)$ has $\sopone$ with the inconsistency witnessed by
$\psi(y_1,y_2)$, then there is an existentially closed model $M$ and tuples
$c_1, c_2, b_1, b_2$ such that $c_1 \nforks^u_M c_2$, $c_1 \nforks^u_M b_1$,
$c_2 \nforks^u_M b_2$, $b_1 \equiv_M b_2$, and such that 
$$\M \models \phi(b_1, c_1) \land \phi(b_2, c_2) \land \psi(c_1,c_2).$$
\end{proposition}
\begin{proof}
By Lemma \ref{skolem-lemma} we may assume that $\edcl(A)$ is an existentially
closed model of $T$ for every subset $A$ of $\M$. 

Let $\lambda$ be a cardinal. By compactness, for arbitrarily large $\kappa$ we
can find parameters $(a_\eta)_{\eta \in 2^{<\kappa}}$ such that
\begin{enumerate}[(i)]
\item for every $\sigma \in 2^\kappa$ the set $\{\phi(x, a_{\sigma|_i}) : i <
\kappa\}$ is consistent;
\item for every $\eta, \nu \in 2^{<\kappa}$ such that $\eta^\frown0 \preceq
\nu$, $\M \models \psi(a_{\eta^\frown1},a_\nu)$.
\end{enumerate}

Construct a sequence $(\eta_i, \nu_i)_{i < \lambda}$ inductively. Assume
$(\eta_j, \nu_j)_{j < i}$ has been constructed. Let $\eta = \bigcup_{j < i}
\eta_i$. If $\kappa$ is large enough, there are $\alpha < \beta < \kappa$ such
that
$$a_{\eta^\frown0^{\alpha \frown}1} \equiv_{\{a_{\eta_j} a_{\nu_j} : j < i\}}
a_{\eta^\frown0^{\beta \frown}1}.$$
Define $\nu_i = \eta^\frown0^{\alpha \frown}1$ and $\eta_i =
\eta^\frown0^{\beta \frown}1$.

Since $\eta_i$ extends $\eta_j$ for $j < i$, there is $b_2$ that realises
$\{\phi(x, a_{\eta_i}) : i < \lambda\}$. If $\lambda$ is large enough, by Fact
\ref{erdos-rado-indiscernible}, there is a sequence $(e_i,d_i)_{i < \omega+2}$
indiscernible over $b_2$, such that for every $n < \omega$ there are $i_1 <
\dots < i_n < \lambda$ with
$$e_1d_1\dots e_nd_n \equiv_{b_2} a_{\eta_{i_1}} a_{\nu_{i_1}} \dots  a_{\eta_{i_n}} a_{\nu_{i_n}}.$$ 
Let $M = \edcl(\{e_nd_n : n < \omega\})$, $c_1 = d_\omega$ and $c_2 =
e_{\omega+1}$. Then $c_1 \nforks^u_{\{e_nd_n : n < \omega\}} c_2$ and $c_2
\nforks^u_{\{e_nd_n : n < \omega\}} b_2$ by indiscernibility. It follows that
$c_1 \nforks^u_M c_2$, $c_2 \nforks^u_M b_2$ and $\M \models \phi(b_2, c_2)$.
Also note that $c_1c_2 = d_\omega e_{\omega+1} \equiv a_{\nu_0} a_{\eta_1}$. But
for some $\alpha$, we have $\nu_0 = {0^\alpha}^\frown1$ and $0^{\alpha+1} \prec
\eta_1$. So $\M \models \psi(c_1,c_2)$.

It remains to find an appropriate $b_1$. We claim that $e_\omega \equiv_M
d_\omega$. Indeed for a fixed $n < \omega$ find $i_0 < \dots  i_n < \lambda$ such
that $e_0d_0\dots e_nd_ne_\omega d_\omega \equiv_{b_2}
a_{\eta_{i_0}}a_{\nu_{i_0}}\dots a_{\eta_{i_n}}a_{\nu_{i_n}}$. It now follows that
\begin{align*}
e_0d_0\dots e_{n-1}d_{n-1}e_\omega & \equiv
a_{\eta_{i_0}}a_{\nu_{i_0}}\dots a_{\eta_{i_{n-1}}}a_{\nu_{i_{n-1}}}a_{\eta_{i_n}}
\\
& \equiv
a_{\eta_{i_0}}a_{\nu_{i_0}}\dots a_{\eta_{i_{n-1}}}a_{\nu_{i_{n-1}}}a_{\nu_{i_n}}
& \text {by the choice of $\eta_{i_n}$ and $\nu_{i_n}$} \\
& \equiv
e_0d_0\dots e_{n-1}d_{n-1}d_\omega.
\end{align*}
Now let $f \in \Aut(\mathfrak M/M)$ such that $f(e_\omega) = d_\omega = c_1$.
Let $b_1 = f(b_2)$. Then $b_1 \equiv_M b_2$ and $\M \models \phi(b_2, e_\omega)$, 
so we also have $\M \models \phi(b_1, c_1)$. Finally $e_\omega \nforks^u_{\{e_nd_n : n < \omega\}}
b_2$ by indiscernibility. Therefore $c_1 \nforks^u_{\{e_nd_n : n < \omega\}}
b_1$ and so $c_1 \nforks^u_M b_1$.
\end{proof}

\end{appendices}

\bibliographystyle{plainnat}
\bibliography{ref}

\end{document}